\documentclass[12pt]{amsart}
\usepackage{amssymb,amsfonts,amsmath,amsopn,amstext,amscd,latexsym, amsthm, enumerate,mathrsfs}

\usepackage{color}
\usepackage{xcolor}
\usepackage[]{hyperref}
\hypersetup{
     colorlinks   = true,
     citecolor    = blue
}

\usepackage[margin=30mm]{geometry}
\usepackage{bbm}
\usepackage[all,cmtip]{xy}

\usepackage{enumerate}
\usepackage[shortlabels]{enumitem}

\newtheorem{theorem}{Theorem}[section]
\newtheorem{thm}[theorem]{Theorem}
\newtheorem{lemma}[theorem]{Lemma}

\newtheorem*{thmA}{Theorem~A}
\newtheorem*{thmB}{Theorem~B}
\newtheorem*{thmC}{Theorem ~C}

\newtheorem{lem}[theorem]{Lemma}

\theoremstyle{definition}

\newtheorem*{hypA}{Hypothesis A}

\theoremstyle{remark}

\DeclareMathOperator{\Real}{Re}

\DeclareMathOperator{\GL}{GL}
\DeclareMathOperator{\SL}{SL}

\DeclareMathOperator{\Mult}{M}

\newcommand{\Layer}{{\mathbf E}}
\newcommand{\RRR}{\mathcal{R}(a,b)}

\newcommand{\bbZ}{{\mathbb Z}}

\newcommand{\Syl}{{\mathrm {Syl}}}

\DeclareMathOperator{\Sp}{Sp}
\DeclareMathOperator{\Sym}{S}

\DeclareMathOperator{\PSL}{PSL}
\DeclareMathOperator{\PSU}{PSU}
\DeclareMathOperator{\PSp}{PSp}
\DeclareMathOperator{\PGL}{PGL}

\newcommand{\Fit}{\mathbf{F}}

\newcommand{\OO}{\mathbf{O}}
\newcommand{\Centralizer}{\mathbf{C}}
\newcommand{\Center}{\mathbf{Z}}
\newcommand{\Normalizer}{\mathbf{N}}

\numberwithin{equation}{section}

\newcommand{\Alt}{{\mathrm {A}}}
\newcommand{\Out}{{\mathrm {Out}}}

\begin{document}

\title[]{$2$-parts of real class sizes}

\author[H. P. Tong-Viet]{Hung P. Tong-Viet}

\address{Department of Mathematical Sciences, Binghamton University, Binghamton, NY 13902-6000, USA}
\email{tongviet@math.binghamton.edu}

\begin{abstract} We investigate the structure of finite groups whose non-central real  class sizes have the same $2$-part. In particular, we prove that such groups  are solvable and have $2$-length one. As a consequence, we show that  a finite group is solvable if it has two real class sizes. This confirms a conjecture due to G. Navarro, L. Sanus and  P. Tiep. 
\end{abstract}

\subjclass[2010]{Primary 20E45; Secondary 20D10}

\date{August 8, 2018}

\keywords{Real conjugacy classes; involutions; $2$-parts; quasi-simple groups}

\thanks{This material is based upon work supported by the National Science Foundation under Grant No. DMS-1440140 while the author was in residence at the Mathematical Sciences Research Institute in Berkeley, California, during the Spring 2018 semester.}

\maketitle

\section{Introduction}
Let $G$ be a finite group. An element $x$ of $G$ is \emph{real} if $x$ and $x^{-1}$ are conjugate in $G$. A conjugacy class $x^G$ of $G$ is real if $x^G$ contains a real element. If $x^G$ is a real class of $G$, then we call the size of $x^G$, denoted by $|x^G|$, a \emph{real class size} and a \emph{non-central real class size},  if moreover, $x$ is non-central in $G$, i.e., $x$ does not  lie in the center $\Center(G)$ of $G$.

A classical result due to Burnside (\cite[Corollary 23.4]{Dornhoff}) states that a finite group is of odd order if and only if the identity element is the only real element. This result has been generalized by Chillag and Mann \cite{CM}, where the authors showed that 
if a finite group $G$ has only one real class size or equivalently every real element lies in $\Center(G)$, then $G$ is isomorphic to a direct product of a $2$-group and a group of odd order. 

The main purpose of this paper is to prove the following.
\begin{thmA}
Let $G$ be a finite group. If $G$ has two real class sizes, then $G$ is solvable.
\end{thmA}

This confirms a conjecture due to G. Navarro, L. Sanus and P. Tiep. Theorem A is best possible in the sense that there are non-solvable groups with exactly three real class sizes.  In fact, the special linear group $\SL_2(q)$ of degree $2$ over a finite field of size $q$, where $q\ge 7$ is a prime power and is congruent to $-1$ modulo $4$, has three real class sizes, namely $1, q(q-1)$ and $q(q+1)$ (see \cite[Theorem 38.1]{Dornhoff}) but $\SL_2(q)$ is non-solvable.

The extremal condition as in Theorem A has been studied extensively in the literature for  conjugacy classes as well as character degrees of finite groups. It\^{o} \cite{Ito1} has shown that if  a finite group has only two class sizes, then it is nilpotent. The alternating group $\Alt_4$ has two real class sizes, which are $1$ and $3$, but it is not nilpotent. 
So, we cannot replace solvability by nilpotency in Theorem A. It\^{o} \cite{Ito2} also showed that if a finite group has three class sizes, then it is solvable. Clearly, this cannot happen for real class sizes  by our example in the previous paragraph.

Recall that a character $\chi$ of a finite group $G$ is \emph{real-valued} if $\chi$ takes real values or equivalently, $\chi$ coincides with its complex conjugate.   Notice that the reality of conjugacy classes of a group can be read off from the character table of the group. This follows from the fact that an element $x\in G$ is real if and only if $\chi(x)$ is real for all complex irreducible characters $\chi$ of $G$ (\cite[Lemma 23.2]{Dornhoff}).

For the corresponding results in real-valued characters, 
Iwasaki \cite{Iwasaki}
and Moret\'{o} and Navarro \cite{MN} have studied the structure of finite groups with two and three real-valued irreducible characters. They show that all these groups must be solvable and their Sylow $2$-subgroups have a very restricted structure. By Brauer's Lemma on character tables (see \cite[Theorem $23.3$]{Dornhoff}), the number of real conjugacy classes and the number of real-valued irreducible characters of a finite group are the same. Thus the aforementioned results also give the structure of finite groups with at most three real conjugacy classes.  For degrees of real-valued characters, Navarro, Sanus and Tiep \cite[Theorem B]{NST} proved that a finite group is solvable if it has at most  three real-valued character degrees.

As already noted in \cite{NST}, any possible  proof of Theorem A is complicated. Instead of giving a direct proof of Theorem A, we will prove a much stronger result which implies Theorem A. For an integer $n\ge 1$ and a prime $p$,  the $p$-part of $n$, denoted by $n_p$ is the largest power of $p$ dividing $n$.

\begin{thmB} 
Let $G$ be a finite group. Suppose that all non-central real class sizes of  $G$ have the same $2$-part.  Then $G$ is solvable.
\end{thmB}

In other words, if $|x^G|_2=2^a$ for all non-central real elements $x\in G$, where $a\ge 0$ is a fixed integer, then $G$ is solvable. 
In fact, we can say more about the structure of these groups. 

\begin{thmC}
Let $G$ be a finite group. Suppose that all non-central real class sizes of  $G$ have the same $2$-part. Then $G$ has $2$-length one.
\end{thmC}
Recall that a group $G$ is said to have $2$-length one if there exist normal subgroups $N\leq K\leq G$ such that $N$ and $G/K$ have odd order and $K/N$ is a $2$-group. Theorem C confirms a conjecture proposed in \cite{Tong}.

Several variations of Theorem B are simply not true. Indeed, if we weaken the hypothesis of Theorem B  by assuming  that $|x^G|_2$ is $1$ or $2^a$ for all real elements $x\in G,$ then $G$ need not be solvable. For example, let $G=\SL_2(2^f)$ with $f\ge 2.$ Then every element of $G$ is real and $|x^G|_2=1$ or $2^f$ for all elements $x\in G.$ We cannot restrict the hypothesis to only real elements of odd order  as $\SL_2(7)$ has only one  conjugacy class of non-central real elements of odd order.
 Also, Theorem B does not hold for odd primes, at least for primes $p$ with $p\equiv -1$ (mod $4$).  In fact, we can take $G=\PSL_2(27)$ if $p=3$ and $G=\SL_2(p)$ if $p\ge 7$. We can check that  all non-central real class sizes of $G$ have the same $p$-part. There are  also some examples for primes $p$ with $p\equiv 1$ (mod $4$), for example,  we can take $G=\PSL_3(3)$ if $p=13$. (We are unable to find any example with $p=5$.)

Theorems B and C  can be considered as (weak) real conjugacy classes versions (only for even prime) of the famous Thompson's theorem on character degrees stating that if a prime $p$ divides the degrees of all non-linear irreducible complex characters of a finite group $G$, then $G$ has a normal $p$-complement. The exact analog of Thompson's theorem does not hold for conjugacy classes. However, Casolo, Dolfi and Jabara \cite{CDJ} proved that for a fixed prime $p$, if all non-central class sizes of a finite group $G$ have the same $p$-part, then $G$ is solvable and has a normal $p$-complement. Some real-valued characters versions of Thompson's theorem were obtained in \cite{NST,NT}.

In order to describe our strategy, we need some terminology and results from abstract group theory which can be found in Chapter $31$ of \cite{Aschbacher}. 
For a finite group $X$, the \emph{layer} of $X$, denoted by $\Layer(X)$, is the subgroup of $X$ generated by all quasi-simple subnormal subgroups (or \emph{components}) of $X$. A  finite group  $L$ is said to be \emph{quasi-simple} if $L$ is perfect and $L/\Center(L)$ is a non-abelian simple group. The \emph{generalized Fitting subgroup} of $X$, denoted by $\Fit^*(X)$, is the central product of $\Layer(X)$ and the Fitting subgroup $\Fit(X)$, the largest nilpotent normal subgroup of $X$. Bender's theorem states that $\Centralizer_X(\Fit^*(X))\leq \Fit^*(X)$ (see, for example, \cite[31.13]{Aschbacher}).

Back to our problem, for a finite group $G$, we denote by $\Real(G)$ the set of all real elements of $G$. Let $G$ be a finite group with $|x^G|_2=2^a$ for all $x\in\Real(G)\setminus \Center(G)$, where $a\ge 0$ is a fixed integer. An important consequence of this hypothesis which is key to our proofs is that if $x\in \Real(G)\setminus\Center(G)$ is a $2$-element and $t\in G$ is a $2$-element inverting $x$, then $\Centralizer_G(\langle x,t\rangle)$ has a normal Sylow $2$-subgroup; in particular, if $i$ is a non-central involution of $G$, then $\Centralizer_G(i)$ has a normal Sylow $2$-subgroup (Lemma \ref{lem: consequences}).  

Let $K=\OO^{2'}(G)$. Then $K$ satisfies the same hypothesis as $G$ does (Lemma \ref{lem:normal odd index}). Let $H$ be the quotient group $K/\OO_{2'}(K)$. Then $H=\OO^{2'}(H)$ and $\OO_{2'}(H)=1$  which implies that $\Fit(H)=\OO_2(H)$ and $\Fit^*(H)=\OO_2(H)\Layer(H)$. We consider two cases according to whether $\Layer(H)$ is trivial or not. When $\Layer(H)=1$, we have $\Fit^*(H)=\OO_2(H)$. Using \cite[$31.16$]{Aschbacher},  $\Fit^*(\Normalizer_H(U))=\OO_2(\Normalizer_H(U))$ for all $2$-subgroups $U\leq H$.  Using this, we can show that $H$ is a $2$-group and Theorem B holds in this case (Theorem \ref{th: 2-constrained}). 
  When $\Layer(H)$ is nontrivial, it must be a quasi-simple group (Lemma \ref{lem: reduction}) and actually it is isomorphic to $\SL_2(q)$ with $q\ge 5$ an odd prime power (Theorem \ref{th: classification}). Using some ad hoc arguments, we finally show that this cannot occur proving Theorem B. For Theorem C, we know that $G$ is solvable by Theorem B. Now by Bender's theorem, we know that $\Fit^*(H)=\OO_2(H)$ and thus $H$ is a $2$-group  by Theorem \ref{th: 2-constrained}. It follows that $G$ has $2$-length one as wanted.

Finally, we would like to point out some connections to other classical results in the literature. If $G$ is a finite group which satisfies the hypothesis of Theorem B, then the centralizer of every non-central involution of $G$  has a normal Sylow $2$-subgroup.  Finite groups in which the centralizers of involutions are $2$-closed (a group is $2$-closed if it has a normal Sylow $2$-subgroup) have been studied by Suzuki in \cite{Suzuki}. These are the so-called (C)-groups. 
On the other hand, the groups in which all involutions are central were classified by R. Griess \cite{Griess}. 
In view of these results, it would be interesting to investigate the structure of finite groups in which in the centralizers of all non-central involutions are $2$-closed.

The paper is organized as follows. In Section \ref{sec:2}, we collect some properties of groups whose non-central real class sizes have the same $2$-part. We prove Theorem C in Section \ref{sec3} and Theorem B in Section \ref{sec4}. In the last section, we classify all finite quasi-simple groups that can appear as composition factors of the groups considered in Section \ref{sec4}.

\section{$2$-parts of non-central real class sizes}\label{sec:2}

In this section, we collect some important properties of real classes as well as draw some consequences on the structure of groups satisfying the hypothesis of Theorem B. Recall that $\Real(G)$ is the set of all real elements of a finite group $G$. We begin with some properties of real elements and real class sizes.


Fix a nontrivial real element $x\in G$ and set $\Centralizer^*_G(x)=\{g\in G\:|\:x^g\in \{x,x^{-1}\}\}.$ Then $\Centralizer^*_G(x)$ is a subgroup of $G$ containing $\Centralizer_G(x)$ as a normal subgroup. If $t\in G$ such that $x^t=x^{-1}$, then $x^{t^2}=x$ so $t^2\in\Centralizer_G(x)$. Assume that $x$ is not an involution. We see that $t\in\Centralizer_G^*(x)\setminus \Centralizer_G(x)$  and if $h\in \Centralizer_G^*(x)\setminus \Centralizer_G(x)$, then $x^h=x^{-1}=x^t$ so that $ht^{-1}\in\Centralizer_G(x)$ or equivalently $h\in\Centralizer_G(x)t$. Thus $\Centralizer_G(x)$ has index $2$ in $\Centralizer_G^*(x)$ and $\Centralizer_G^*(x)=\Centralizer_G(x)\langle t\rangle$. Hence $|x^G|$ is even whenever $x$ is a real element with $x^2\neq 1$. In particular, $|x^G|$ is even if $x$ is a nontrivial real element of odd order. 
Notice that if $x\in \Real(G)\cap\Center(G)$, then $x^2=1$.

\begin{lem}\label{lem: real elements}  Let $G$ be a finite group and let $N\unlhd G$.

\begin{enumerate}[$(1)$]
\item If $x\in G$ is real, then every power of $x$ is also real.

\item If $x^g=x^{-1}$ for some $x,g\in G$, then $x^t=x^{-1}$ for some $2$-element $t\in G$.
\item If $x\in \Real(G)$ and $|x^G|$ is odd, then $x^2=1.$
\item If $|G:N|$ is odd, then $\Real(G)=\Real(N)$.
\end{enumerate}
\end{lem}

\begin{proof} Let $x\in G$ be a real element. Then $x^g=x^{-1}$ for some $g\in G$. If $k$ is any integer, then $(x^k)^g=(x^g)^k=(x^{-1})^k=(x^{k})^{-1}$ so $x^k$ is real which proves (1). Write $o(g)=2^am$ with $(2,m)=1$ and let $t=g^m$. Then $t$ is a $2$-element and $x^t=x^{g^m}=x^g=x^{-1}$ as $g^2\in\Centralizer_G(x)$  and $m$ is odd. This proves (2). Part (3) follows from the discussion above.
For part (4), suppose that $N\unlhd G$ and $|G/N|$ is odd. Let $x\in\Real(G)$. Then there exists a $2$-element $t\in G$ inverting $x$ by (2). As $|G/N|$ is  odd and $t$ is a $2$-element, $t\in N$ and thus $x^{-2}=t^{-1}(xtx^{-1})\in N$. Again, as $|G/N|$ is odd, $x\in N$ and so $x\in\Real(N)$.
\end{proof}

The first two claims of the following lemma is well-known. The last claim follows from \cite[Proposition~$6.4$]{DNT} whose proof uses  Baer-Suzuki theorem.

\begin{lemma}\label{lem: no odd} Let $G$ be a finite group and let $N\unlhd G$. Then

\begin{itemize}
\item[$(1)$] If $x\in N$, then $|x^N|$ divides $|x^G|$.
\item[$(2)$] If $Nx\in G/N$, then $|(Nx)^{G/N}|$ divides $|x^G|$.
\item[$(3)$]  $G$ has no nontrivial real element of odd order if and only if $G$ has a normal Sylow $2$-subgroup.
\end{itemize}
\end{lemma}

We next study the structure of finite groups  in which all real $2$-elements are central. 
The proof of the following result  is similar to that of  \cite[Corollary C]{IN2}.

\begin{lem}\label{lem: central real 2-elements}
Let $G$ be a finite group. Assume that all real $2$-elements of $G$ lie in $\Center(G)$. Then $G$ has a normal $2$-complement.
\end{lem}

\begin{proof} 
Let $P$ be a $2$-subgroup of $G$. Then $\Real(P)\subseteq \Center(G)$ by the hypothesis of the lemma.   Let $Q\leq \Normalizer_G(P)$ be a $2'$-group. Since $\Real(P)\subseteq \Center(G)$, $Q$ centralizes all real elements of $P$,  \cite[Theorem B]{IN2} now implies that $Q$ centralizes $P$. Hence $Q\leq \Centralizer_G(P)$. It follows that $\Normalizer_G(P)/\Centralizer_G(P)$ is a $2$-group. As $P$ is chosen arbitrarily,  Frobenius normal $p$-complement theorem (\cite[39.4]{Aschbacher}) implies that $G$ has a normal $2$-complement.
\end{proof}


Let $G$ be a finite group with $|G|_2=2^{a+b}$, where $a,b\ge 0$ are fixed integers. Observe that
the two conditions $|x^G|_2=2^a$ for all  elements $x\in \Real(G)\setminus\Center(G)$  and $|\Centralizer_G(x)|_2=2^b$ for all $x\in \Real(G)\setminus\Center(G)$ are equivalent.
For brevity, we say that a finite group $G$ is an \emph{$\RRR$-group} if $G$ satisfies one of the two equivalent conditions above. 

\begin{lem}\label{lem: consequences} Let $G$ be an $\RRR$-group.  Let $x\in \Real(G)\setminus\Center(G)$ and  let $t\in G$ be a $2$-element such that $x^t=x^{-1}$. Then
\begin{itemize}
\item[$(a)$] If  $o(x)=2m$ for some integer $m>1$, then $x^m\in\Center(G)$.
\item[$(b)$] If $x$ is a $2$-element, then  $\Centralizer_G(\langle x,t\rangle)$ is $2$-closed. Moreover, if $x$ is an involution of $G$, then $\Centralizer_G(x)$ is $2$-closed.
\end{itemize}

\end{lem}

\begin{proof}  
For part (a), suppose that $o(x)=2m$  with $m>1$. 
Notice that $x^m$ is an involution, $t\in\Centralizer_G^*(x)\setminus\Centralizer_G(x)$ and $|\Centralizer_G(x)|_2=2^b$.  We have $(x^m)^t=(x^{m})^{-1}=x^m$, so $t\in\Centralizer_G(x^m)$ and hence $\Centralizer_G(x)\leq \Centralizer_G^*(x)\leq \Centralizer_G(x^m)$. Since $x^2\neq 1,$ we have $|\Centralizer_G^*(x)|_2=2|\Centralizer_G(x)|_2=2^{b+1}$. Therefore $|\Centralizer_G(x^m)|_2\ge |\Centralizer_G^*(x)|_2>2^b$. Since $G$ is a $\RRR$-group, this forces $x^m\in\Center(G)$.

For part (b), assume  that $x$ is a real $2$-element inverted by $t$. Let $J:=\Centralizer_G(\langle x,t\rangle)$. We first claim that $J$ has no nontrivial real element of odd order. By contradiction, suppose that there exist $y,s\in J$ with $y^s=y^{-1}$, where $o(y)>1$ is odd. Observe that $[x,y]=[x,s]=[t,y]=[t,s]=1$ and $x^t=x^{-1},y^s=y^{-1}$ so $$(xy)^{st}=x^{st}y^{st}=x^ty^{ts}=x^ty^s=x^{-1}y^{-1}=y^{-1}x^{-1}=(xy)^{-1}.$$ So $xy\in\Real(G)\setminus\Center(G)$. 
Since $(o(x),o(y))=1$,  $\Centralizer_G(xy)=\Centralizer_G(x)\cap\Centralizer_G(y)=\Centralizer_A(y)$, where $A=\Centralizer_G(x)\ge J.$ Let $U$ be a Sylow $2$-subgroup of $\Centralizer_G(xy)$. Then $U$ is also a Sylow $2$-subgroup of both $\Centralizer_G(x)$  and $\Centralizer_G(y)$ as  $|\Centralizer_G(xy)|_2=|\Centralizer_G(x)|_2=|\Centralizer_G(y)|_2=2^b$ (noting $x,y,xy\in\Real(G)\setminus\Center(G)$). Thus $\Centralizer_A(y)$ contains a Sylow $2$-subgroup of $A$, therefore $|y^A|$ is odd and hence $y^2=1$ by Lemma \ref{lem: real elements}(3). (Note that $y\in\Real(J)\subseteq\Real(A)$.) However, as $o(y)$ is odd, $y=1$, which is a contradiction.
Thus $J$ has no nontrivial real element of odd order and so by Lemma \ref{lem: no odd}(3) $J$ has a normal Sylow $2$-subgroup as required. Finally, if $x$ is an involution, then we can choose $t=1$ and the result follows.
\end{proof}

In the next lemma, we show that every normal subgroup of odd index of an $\RRR$-group is also an $\RRR$-group.
\begin{lem}\label{lem:normal odd index}
Suppose that $G$ is an $\RRR$-group. Let $K$ be a normal subgroup of $G$ of odd index. Then $K$ is also an $\RRR$-group.
\end{lem}

\begin{proof}
Let $K$ be a normal subgroup of $G$ of odd index. By Lemma \ref{lem: real elements}(4), we have $\Real(G)=\Real(K)$. Let $x\in \Real(K)\setminus\Center(K)$ and let $C:=\Centralizer_G(x)$. Let $P\in\Syl_2(C)$ and let $S\in\Syl_2(G)$ such that $P\leq S$. Since $|G:K|$ is odd, we have $P\leq S\leq K$. In particular, $S\in\Syl_2(K)$. We have  $P\leq K\cap C=\Centralizer_K(x)\leq C$. Thus $P$ is also a Sylow $2$-subgroup of $\Centralizer_K(x)$. Therefore $|C|_2=|\Centralizer_K(x)|_2$ and hence $K$ is an $\RRR$-group.
\end{proof}

The first part of the following lemma is  essentially Lemma 2.2 in \cite{GNT} and its proof. For the reader's convenience, we include the proof here.

\begin{lem}\label{lem:centralizers of 2-groups}
Let $G$ be a finite group and let $N\unlhd G$ be a normal subgroup of odd order. Let $xN\in G/N$ be a  real element. Write $\overline{G}=G/N$.

\begin{itemize}
\item[$(1)$] There exist $y\in G$  and a $2$-element $t\in G$ such that  $\overline{x}=\overline{y}$  and $y^t=y^{-1}$.
\item[$(2)$]  If $\overline{x}$ is a $2$-element, then $y$ can be chosen to be a $2$-element. Moreover, if $\overline{x}$ is an involution in $\overline{G}$, then $y$ can be chosen to be an involution in $G$.

\item[$(3)$] In Claim $(2)$,  we have $|\overline{x}^{\overline{G}}|_2=|y^G|_2$ and $\Centralizer_{\overline{G}}(\langle \overline{x},\overline{t}\rangle)=\overline{\Centralizer_G(\langle y,t\rangle)}$.

\end{itemize}
\end{lem}

\begin{proof} Let $xN$ be a real element of $G/N$. Then there exists a $2$-element $tN\in G/N$ inverting $xN$ by Lemma \ref{lem: real elements}(2). By considering the $2$-part of $t$, we can assume that $t$ is a $2$-element. The maps $z\mapsto z^{-1}$ and $z\mapsto z^t$ are commuting permutations of $G$  having 2-power order and thus their product $\sigma$ has $2$-power order.  Each of these two permutations maps $xN$ to $x^{-1}N$ and $x^{-1}N$ to $xN$. So $\sigma$ defines a permutation on $xN$. Since $\sigma$ has $2$-power order and $|xN|=|N|$ is odd, $\sigma$ has a fixed point $y\in xN$. Thus $y=y^{\sigma}=(y^t)^{-1}$. So $y^t=y^{-1}$ as wanted. This proves $(1)$.

 Assume next that the order of $xN$ in $G/N$ is $2^k$ for some integer $k\ge 0.$ By part (1), there exist elements $y,t\in G$  such that $xN=yN$ and $y^t=y^{-1}$. Since $|N|$ is odd,  $o(y)=2^km$, where $m$ is odd and $y^{2^k}\in N$. Since $(2^k,m)=1$, there exist integers $u,v$ such that $1=um+v2^k$ and thus $yN=(y^{um}N)(y^{v2^k}N)=y^{um}N$. Clearly, $y^{um}$ is a $2$-element and is inverted by $t$. Replace $y$ by $y^{um}$, the first claim of part (2) follows. 
 Moreover, as $|N|$ is odd,  if $y^2\in N$ and $y$ is a $2$-element, then $y^2=1$. Hence the last claim of $(2)$ follows.
 
 Finally, for part (3), by  \cite[Lemma 7.7]{Isaacs-1}, $\Centralizer_{\overline{G}}(\overline{x})=\Centralizer_{\overline{G}}(\overline{y})=\overline{\Centralizer_G(y)}$  since $(o(y),|N|)=1$. If $U$ is a Sylow $2$-subgroup of $\Centralizer_G(y)$, then $\overline{U}$ is a Sylow $2$-subgroup of $\overline{\Centralizer_G(y)}$ and $|U|=|\overline{U}|$ so $|\overline{y}^{\overline{G}}|_2=|y^G|_2$. Finally, we have  $\Centralizer_{\overline{G}}(\langle \overline{x},\overline{t}\rangle)=\Centralizer_{\overline{G}}(\langle \overline{y},\overline{t}\rangle)=\overline{\Centralizer_G(\langle y,t\rangle)}$, where the second equality follows from  \cite[Lemma 7.7]{Isaacs-1} again as $\langle y,t
\rangle$ is a $2$-group. 
\end{proof}

 In the next lemma, we determine some properties of the quotient group $G/\OO_{2'}(G)$.

\begin{lem}\label{lem:odd quotient consequences}
Let $G$ be an $\RRR$-group and let $T$ be a subgroup of $G$ containing $\OO_{2'}(G)$. Let $\overline{G}=G/\OO_{2'}(G)$. Then the following hold.

\begin{itemize}
\item[$(1)$] If $\overline{x}\in \Real(\overline{T})\setminus\Center(\overline{T})$ is a $2$-element, then there exists  a $2$-element $\overline{t}\in \overline{T}$ inverting $\overline{x}$ and $\Centralizer_{\overline{G}}(\langle\overline{x},\overline{t}\rangle)$ has a normal Sylow $2$-subgroup. 

\item[(2)] If $\overline{x}\in \Real(\overline{T})\setminus\Center(\overline{T})$  is an involution, then $\Centralizer_{\overline{G}}(\overline{x})$ is $2$-closed.

\item[$(3)$] If  $\overline{z},\overline{x}\in\Real(\overline{G})\setminus \Center(\overline{G})$ and $\overline{x}$ is a $2$-element, then $|\overline{z}^{\overline{G}}|_2\leq |\overline{x}^{\overline{G}}|_2=2^a$ and $|\Centralizer_{\overline{G}}(\overline{x})|_2=2^b$.   
 \end{itemize}
\end{lem}

\begin{proof}
For part (1), let $\overline{x}\in\Real(\overline{T})\setminus\Center(\overline{T})$ be a $2$-element. By applying Lemma \ref{lem:centralizers of 2-groups} for $T$, there exist $2$-elements $y,t\in T$ such that $\overline{x}=\overline{y}$ and $y^{t}=y^{-1}$. Since $\overline{y}=\overline{x}\not\in\Center(\overline{T})$, $y\not\in \Center(T)$ and so  $y\in\Real(G)\setminus\Center(G)$.
   By Lemma \ref{lem: consequences}(b), $\Centralizer_G(\langle y,t\rangle)$ is $2$-closed and thus by  Lemma 
\ref{lem:centralizers of 2-groups}(3), $\Centralizer_{\overline{G}}(\langle\overline{x},\overline{t}\rangle)=\Centralizer_{\overline{G}}(\langle \overline{y},\overline{t}\rangle)=\overline{\Centralizer_G(\langle y,t\rangle)}$ is $2$-closed.  

Part (2) follows from Lemma \ref{lem:centralizers of 2-groups}(2) and part (1) above. 
For part (3), let $\overline{z},\overline{x}\in \Real(\overline{G})\setminus\Center(\overline{G})$, where $\overline{x}$ is a $2$-element. By Lemma \ref{lem:centralizers of 2-groups}(1), there exists  $w\in \Real(G)\setminus\Center(G)$ with $\overline{w}=\overline{z}$. We have that $|\overline{w}^{\overline{G}}|_2\leq |w^G|_2=2^a$ as $|\overline{w}^{\overline{G}}|$ divides $|w^G|$.
By Lemma \ref{lem:centralizers of 2-groups}(3), $|\overline{x}^{\overline{G}}|_2=|y^G|_2=2^a$, which implies that $|\Centralizer_{\overline{G}}(\overline{x})|_2=2^b$. The proof is now complete.
\end{proof}

\section{Proof of Theorem C}\label{sec3}
 In this section, we will prove Theorem C assuming the solvability from Theorem B. Indeed, Theorem C will follow immediately from Theorem \ref{th: 2-constrained}. It is convenient to state the following.

\begin{hypA}\label{hyp:A}  Let $G$ be a finite group with $\OO_{2'}(G)=1$ and $G=\OO^{2'}(G)$. Let $a,b\ge 0$ be fixed integers. Assume the following hold.

\begin{itemize}
\item[(1)] If $T\leq G$ and $x\in \Real(T)\setminus\Center(T)$ is a $2$-element, then there exists a $2$-element $t\in T$ inverting $x$ and $\Centralizer_G(\langle x,t\rangle)$  is $2$-closed.
\item[(2)]  If $x\in G\setminus\Center(G)$ is an involution, then $\Centralizer_G(x)$ is $2$-closed.

\item[$(3)$] If $x\in \Real(G)\setminus\Center(G)$  is a $2$-element, then $|x^G|_2=2^a$ and $|\Centralizer_G(x)|_2=2^b$.

\item[$(4)$] If $z\in \Real(G)\setminus\Center(G)$, then $|z^G|_2\le 2^a$ and $|\Centralizer_G(z)|_2\ge 2^b$.
\end{itemize}
\end{hypA}

Using Lemmas \ref{lem:normal odd index} and \ref{lem:odd quotient consequences},  we show that if a finite group $G$ satisfies the hypothesis of Theorem B, then certain section of $G$ satisfies the hypothesis above.
\begin{lem}\label{lem:hypA} If $G$ is an $\RRR$-group,  $K=\OO^{2'}(G)$ and $H=K/\OO_{2'}(K)$, then $H$ satisfies Hypothesis~\href{hyp:A}{A}.
\end{lem}

\begin{proof}
Let $G$ be an $\RRR$-group and let $K=\OO^{2'}(G)$. We know from Lemma \ref{lem:normal odd index} that $K$ is also an $\RRR$-group. Let $H=K/\OO_{2'}(K)$. Then $\OO^{2'}(H)=H$, $\OO_{2'}(H)=1$ and $H$ satisfies the conclusion of Lemma \ref{lem:odd quotient consequences}. Therefore, $H$ satisfies Hypothesis~\href{hyp:A}{A}.
\end{proof}

The following is an easy consequence of Lemma \ref{lem: central real 2-elements}.

\begin{lem}\label{lem: centralizers of 2-subgroups} Let $G$ be a finite group. Assume that  $|\Centralizer_G(u)|_2=2^b$ for all $2$-elements $u\in \Real(G)\setminus\Center(G)$. If $U\leq G$ is a $2$-subgroup of $G$ with $|U|\ge 2^b$, then $\Centralizer_G(U)$ has a normal $2$-complement.
\end{lem}

\begin{proof}
Let  $U\leq G$ be a $2$-group with $|U|\ge 2^b$. Let $C=\Centralizer_G(U)$. Observe that if all $2$-elements in $\Real(C)$ lie in $\Center(C)$, then $C$ has a normal $2$-complement by Lemma \ref{lem: central real 2-elements}. Thus by contradiction, assume that there exists a real $2$-element $y\in \Real(C)\setminus\Center(C)$. Clearly $y\in\Real(G)\setminus\Center(G)$ and so $|\Centralizer_G(y)|_2=2^b$. We see that $U\leq \Centralizer_G(y)$ as $y\in\Centralizer_G(U)$, so $\langle U,y\rangle$ is a $2$-subgroup of $\Centralizer_G(y)$.   Since $|\Centralizer_G(y)|_2=2^b$ and $|U|\ge 2^b$, $U\in\Syl_2(\Centralizer_G(y))$ which implies that $y\in U$. But this implies that $y\in\Center(C)$ since $[C,U]=1$ and $y\in U$,  a contradiction.
\end{proof}

We now prove the key result of this section.
\begin{thm}\label{th: 2-constrained} Let $G$ be a finite group satisfying Hypothesis~\href{hyp:A}{A}. If $\Fit^*(G)=\OO_2(G)$, then $G$ is a $2$-group.
\end{thm}

\begin{proof} Suppose that $G$ satisfies Hypothesis~\href{hyp:A}{A} and  $\Fit^*\-(G)=\OO_2(G)$ but $G$ is not a $2$-group. Notice that  $G=\OO^{2'}(G)$. If  $G$ has no non-trivial real element of odd order, then $G$ has normal Sylow $2$-subgroup by Lemma \ref{lem: no odd}(3), but then since $G=\OO^{2'}(G)$, it must be a $2$-group. Therefore, we may assume that there exists an element $z\in\Real(G)\setminus\Center(G)$ of odd order.

Let  $V\in\Syl_2(\Centralizer_G(z))$. By Hypothesis~\href{hyp:A}{A}(4),  $|V|\ge 2^b$. Clearly, by  Hypothesis~\href{hyp:A}{A}(3), $G$ satisfies the hypothesis of Lemma \ref{lem: centralizers of 2-subgroups} and so $\Centralizer_G(V)$ has a normal $2$-complement, say $W$. Then $W=\OO_{2'}(\Centralizer_G(V))$ and $\Centralizer_G(V)/W$ is a $2$-group. Now, we see that $W\unlhd \Centralizer_G(V)\unlhd \Normalizer_G(V)$ and $W$ is characteristic in $\Centralizer_G(V)$ so $W\leq \OO_{2'}(\Normalizer_G(V))$. However, by \cite[$31.16$]{Aschbacher}, as $\Fit^*(G)=\OO_2(G)$, we have $\Fit^*(\Normalizer_G(V))=\OO_2(\Normalizer_G(V))$, which forces $\OO_{2'}(\Normalizer_G(V))=1$ (by Bender's theorem \cite[$31.13$]{Aschbacher}),  hence $W=1$ so $\Centralizer_G(V)$ is a $2$-group, which is impossible since $z\in\Centralizer_G(V)$  is a nontrivial element of odd order. This contradiction shows that $G$ is a $2$-group.
\end{proof}

Assuming the solvability from Theorem B, we can now prove Theorem C, which is included in the following.

\begin{thm}\label{th:Theorem C}
Let $G$ be an $\RRR$-group. Then $\OO^{2'}(G)$ has a normal $2$-complement $N$. So $G$ has $2$-length one. Moreover, the $2$-group $\OO^{2'}(G)/N$ has  at most two real class sizes.
\end{thm}

\begin{proof}
By Theorem B, we assume that $G$ is solvable. Let $K=\OO^{2'}(G)$ and $H:=K/\OO_{2'}(K)$. Then $H$ satisfies Hypothesis~\href{hyp:A}{A} by Lemma \ref{lem:hypA}. As $H$ is solvable and $\OO_{2'}(H)=1$, $\Fit^*(H)=\Fit(H)=\OO_2(H)$. By Theorem \ref{th: 2-constrained}, $H$ is a $2$-group and thus $K$ has a normal $2$-complement $N=\OO_{2'}(K)$. So $G$ has $2$-length one.
Finally, as $H$ is $2$-group which satisfies Hypothesis~\href{hyp:A}{A}, for any $x\in\Real(H)$, we have $|x^H|=1$ or $2^a$. 
\end{proof}

\section{Proof of Theorem B}\label{sec4}

We will prove Theorem B in this section. We first prove some reduction results.

\begin{lem}\label{lem: reduction} Let $G$ be a finite non-solvable group satisfying Hypothesis~\href{hyp:A}{A}. Let $\overline{G}=G/\OO_2(G)$. Then $L=\Layer(G)$ is a quasi-simple group whose center is a $2$-group and $\overline{G}$ is an almost simple group with socle $\overline{L}\cong L/\Center(L)$.
\end{lem}

\begin{proof} Since $G$ satisfies Hypothesis~\href{hyp:A}{A}, $G=\OO^{2'}(G)$ and $ \OO_{2'}(G)=1$. We have $\Fit^*(G)=\OO_2(G)\Layer(G)$. As $G$ is non-solvable,  $\Layer(G)$ is non-trivial by Theorem \ref{th: 2-constrained}. Let $L$ be a component of $G$, that is, $L$ is a perfect quasi-simple subnormal subgroup of $G$. Recall that $\Layer(G)$ is generated by all components of $G$.

We first claim that $L=\Layer(G)$. Suppose by contradiction that $G$ has another component, say $L_1\neq L$.
Since $L$ is non-solvable, by Lemma \ref{lem: central real 2-elements} there exists a real $2$-element $x\in\Real(L)\setminus\Center(L)$. By Hypothesis~\href{hyp:A}{A}(1), there exists a $2$-element $t\in L$ such that $x^t=x^{-1}$ and $\Centralizer_G(\langle x,t\rangle)$ has a normal Sylow $2$-subgroup, so it is solvable. However, by \cite[31.5]{Aschbacher}, $[L_1,L]=1$ and since $\langle x,t\rangle\leq L$, $L_1\leq \Centralizer_G(\langle x,t\rangle)$, which is impossible. Thus $L$ is the unique component of $G$ and so $L=\Layer(G)$. As $\OO_{2'}(G)=1$, $\Center(L)$ is a $2$-group.

Let $C=\Centralizer_G(L)$. We  claim that $C= \OO_2(G)$ which implies that $\overline{G}$ is almost simple with socle $\overline{L}$. In fact, we will show that every real $2$-element of $C$ lies in $\Center(C)$ and so by Lemma \ref{lem: central real 2-elements}, $C$ has a normal $2$-complement which is $\OO_{2'}(C)$. Since $L\unlhd G$, $C\unlhd G$ and so $\OO_{2'}(C)\leq \OO_{2'}(G)=1$, hence $C$ is a $2$-group and thus $C\leq\OO_2(G)$. Furthermore, by \cite[31.12]{Aschbacher} $[\OO_2(G),L]=1$ so $\OO_2(G)\leq C$. Therefore, $C=\OO_2(G)$ as wanted. 

To finish the proof, suppose by contradiction that  there exists a real $2$-element $y\in\Real(C)$ which is not in $\Center(C)$. By Hypothesis~\href{hyp:A}{A}(1), there exists a $2$-element $t\in C$ inverting $y$ and $\Centralizer_G(\langle y,t\rangle)$ has a normal Sylow $2$-subgroup. As $[L,C]=1$ and  $\langle y,t\rangle\leq C$,  $L\leq \Centralizer_G(\langle y,t\rangle)$, which is impossible. This completes our proof.
\end{proof}

In order to classify all the possible finite quasi-simple groups which can appear as $\Layer(G)$ in the previous lemma, we need the following.
\begin{lemma}\label{lem:Structure of H} Let $G$ be a finite non-solvable group satisfying Hypothesis~\href{hyp:A}{A}. Let $L=\Layer(G)$ and $\overline{G}=G/\OO_2(G)$. If $x,z\in \Real(L)\setminus \Center(L)$  and $x$ is a $2$-element, then 
 \begin{equation}\label{eqn}|z^L|_2\leq |z^G|_2\leq 2^a=|x^G|_2\leq |\overline{G}:\overline{L}|_2\cdot |x^L|_2\leq |\Out(\overline{L})|_2\cdot |x^L|_2.\end{equation}
\end{lemma}

\begin{proof}  Notice that $L=\Layer(G)$ is a quasi-simple group by Lemma \ref{lem: reduction} and by \cite[$31.12$]{Aschbacher}, $[\OO_2(G),L]=1$. 
Let $x,z\in \Real(L)\setminus\Center(L)$, where $x$ is  a $2$-element. By Hypothesis~\href{hyp:A}{A}(3,4), $|z^G|_2\leq 2^a=|x^G|_2$. Since $L\unlhd G$, $|z^L|_2\leq |z^G|_2\leq 2^a$. 
For the remaining inequalities,  observe that $\OO_2(G)\leq \Centralizer_G(x)$ and $$|x^G|=|G:L\Centralizer_G(x)|\cdot |L\Centralizer_G(x):\Centralizer_G(x)|.$$ We have $|L\Centralizer_G(x):\Centralizer_G(x)|=|L:\Centralizer_L(x)|=|x^L|$ and $$|G:L\Centralizer_G(x)|=|G:L\OO_2(G)|/|\Centralizer_G(x):\Centralizer_{G}(x)\cap L\OO_2(G)|.$$ As $|G:L\OO_2(G)|=|\overline{G}:\overline{L}|$ divides $|\Out(\overline{L})|$, by taking the $2$-parts, we obtain $$|x^G|_2\leq |\overline{G}:\overline{L}|_2\cdot |x^L|_2\leq |\Out(\overline{L})|_2\cdot |x^L|_2.$$
The proof is now complete.
\end{proof}

 Using the classification of finite simple groups, we can show that $L$ in the previous lemma is isomorphic to  a special linear group $\SL_2(q)$, where $q\ge 5$ is an odd prime  power. We  defer the proof of the following theorem until next section. 

\begin{thm}\label{th: classification}
Let $L$ be a quasi-simple group with center $Z$ and let $S=L/Z$. Suppose that $Z$ is a $2$-group and the following conditions hold.

\begin{itemize}
\item[$(1)$] If $i$ is a non-central involution of $L$, then $\Centralizer_L(i)$ is $2$-closed.
\item[$(2)$] If $x,z\in L$ are non-central real elements and $x$ is a $2$-element, then $|z^L|_2\leq |\Out(S)|_2\cdot |x^L|_2$.
\end{itemize}
Then $L\cong\SL_2(q)$ with $q\ge 5$ odd.
\end{thm}
Let $p$ be a prime. If $n\ge 1$ is an integer, then the $p$-adic valuation of $n$, denoted by $\nu_p(n)$, is the highest exponent $\nu$ such that  $p^\nu$  divides $n$. Hence $n_p=p^{\nu_p(n)}$. Notice that $\nu_p(xy)=\nu_p(x)+\nu_p(y)$ for all integers $x,y\ge 1.$ 

The following number theoretic result is obvious.

\begin{lem}\label{lem:2-adic} Let $m,k\ge 1$ be integers, where $m\ge 3$ is odd. Then $\nu_2(m^{2^k}-1)\ge k+2$.
\end{lem}

\begin{proof} We proceed by induction on $k\ge 1.$ Suppose that $m\equiv \epsilon$ (mod $4$), where $\epsilon=\pm 1.$ For the base case, assume $k=1$. Clearly, $m-\epsilon$ is divisible by $4$ while $m+\epsilon$ is divisible by $2$. So $\nu_2(m^2-1)\ge 3=k+2.$

Assume that $\nu_2(m^{2^t}-1)\ge t+2$ for some integer $t\ge 1.$ We have $m^{2^{t+1}}-1=(m^{2^t}-1)(m^{2^t}+1)$. Since $m$ is odd, $\nu_2(m^{2^t}+1)\ge 1$ and $\nu_2(m^{2^t}-1)\ge t+2$ by induction hypothesis. Therefore,  $\nu_2(m^{2^{t+1}}-1)=\nu_2(m^{2^t}+1)+\nu_2(m^{2^t}-1)\ge 1+(t+2)=(t+1)+2$. By the principle of mathematical induction, the lemma follows. 
\end{proof}

We are now ready to prove Theorem B, which we will restate here.
\begin{thm}\label{th: 2-parts}
If $G$ is an $\RRR$-group, then $G$ is solvable.
\end{thm}

\begin{proof}
Let $G$ be a counterexample to the theorem of minimal order. Then $G$ is non-solvable and by Lemma \ref{lem:hypA}, $H$ satisfies Hypothesis~\href{hyp:A}{A}, where $H=K/\OO_{2'}(K)$ and $K=\OO^{2'}(G)$. 
As $G$ is non-solvable, $H$ is non-solvable. Let $L=\Layer(H)$. By Lemmas \ref{lem: reduction}, \ref{lem:Structure of H} and Hypothesis~\href{hyp:A}{A}(2), $L$ is a quasi-simple group  satisfying the hypothesis of Theorem \ref{th: classification}, so $L\cong \SL_2(q)$ with $q\ge 5$ an odd prime power. Moreover,  $\OO^{2'}(\overline{H})=\overline{H}$ and $\overline{H}$ is almost simple with socle $\overline{L}$, where $\overline{H}=H/\OO_2(H)$.

Write $q=p^f\ge 5$, where $p>2$ is a prime and $f\ge 1$. It is well-known that $\Out(\PSL_2(q))\cong \langle \delta\rangle\times\langle \varphi\rangle\cong \bbZ_2\times \bbZ_f$, where $\delta$ is a diagonal automorphism of order $2$ and $\varphi$ is a field automorphism of order $f$ of $\PSL_2(p^f)$. Assume $q\equiv \eta$ (mod $4$) with $\eta=\pm 1$. We have $k:=\nu_2(q-\eta)\ge 2$, $\nu_2(q+\eta)=1$ and $|L|_2=(q^2-1)_2=2^{k+1}$.

Now $L$ has two non-central real elements $z$ and $x$ (which lie in the cyclic subgroups of $L$ of order $q+\eta$ and $q-\eta$) of  order $(q+\eta)/2$ and  $2^k>2$, respectively. It is easy to check that $|x^L|_2=2$ and $|z^L|_2=2^{k}$. Moreover, $z^L$ is invariant under the diagonal automorphisms of $L$. As $\overline{H}/\overline{L}$ is a  subgroup of $\Out(\overline{L})\cong \bbZ_2\times\bbZ_f$ with $\OO^{2'}(\overline{H}/\overline{L})=\overline{H}/\overline{L}$,  $\overline{H}/\overline{L}$ is an abelian $2$-group of order at most $2^{c+1}$, where $c=\nu_2(f)$.

(a) Assume that $f$ is odd. Then $\PSL_2(q)\cong \overline{L}\unlhd \overline{H}\leq \PGL_2(q)$.

If $\overline{H}\cong \PSL_2(q)$, then $H=L\OO_2(H)$. In this case, we see that $|z^H|_2=|z^L|_2=2^k>2=|x^L|_2=|x^H|_2$ violating Equation \eqref{eqn}.

Assume that $\overline{H}\cong\PGL_2(q)$. From the character table of $\PGL_2(q)$ (see \cite[Table III]{Steinberg}), $\overline{H}\cong \PGL_2(q)$ contains a real element $\overline{y}$ of order $p$ (labelled by $A_2$) with $|\overline{y}^{\overline{H}}|=q^2-1$. Since $o(\overline{y})=p$ is odd, by \cite[Lemma $2.2$]{GNT} there exists a real element $w\in H$ of order $p$ such that $\overline{w}=\overline{y}$. Now $|w^H|_2\ge |\overline{y}^{\overline{H}}|_2 = 2^{k+1}>4=|\Out(\overline{L})|_2\cdot |x^L|_2$, violating  Equation \eqref{eqn}.

(b) Assume $f$ is even.  We have that $q\equiv 1$ (mod $4$) and since  $p$ is odd, by Lemma \ref{lem:2-adic} we have $k=\nu_2(q-1)\ge c+2$. Recall that $c=\nu_2(f)$.

From \cite[Theorem 38.1]{Dornhoff}, $L\cong \SL_2(q)$ has a real element $$y=\left(\begin{array}{cc}1 & 0 \\1 & 1\end{array}\right)$$ of order $p$ (labelled by $c$). The image $\overline{y}$ of $y$ in $\overline{L}\cong \PSL_2(q)$  is also a real element of order $p$ and $|y^L|=|\overline{y}^{\overline{L}}|=(q^2-1)/2$. Hence $|y^H|_2\ge |\overline{y}^{\overline{L}}|_2=2^k$ as $|\PSL_2(q)|_2=2^k$.

Assume that $|\overline{H}/\overline{L}|_2\leq 2^c$. We have $|\overline{H}:\overline{L}|_2\cdot |x^L|_2\leq 2^{c+1}<2^{c+2}\leq 2^k\leq |y^H|_2$, contradicting Equation \eqref{eqn}.

Finally, we assume that $|\overline{H}/\overline{L}|_2=2^{c+1}$ and so $\overline{H}/\overline{L}=\langle \delta\rangle\times\langle\varphi^m\rangle$ with $m=f/2^c$. We know that $\overline{y}\in\overline{L}$ is $\varphi$-invariant but not $\delta$-invariant, thus $|\overline{y}^{\overline{H}}|_2=2^{k+1}$ and hence $|y^H|_2\ge |\overline{y}^{\overline{H}}|_2=2^{k+1}\ge 2^{c+3}$. Now $|\Out(\overline{L})|_2\cdot |x^L|_2=2^{c+2}<2^{c+3}\leq |y^H|_2$, which violates Equation \eqref{eqn} again.
The proof is now complete.
\end{proof}

\section{Quasi-simple groups}\label{sec:5}
The main purpose of this section is to prove Theorem \ref{th: classification}.  We first prove some easy results which will be needed in our classification.

For a prime $p$ and a group $X$, the $p$-rank of $X$, denoted by $r_p(X)$, is the maximum rank of an elementary abelian $p$-subgroup of $X$. Recall that the $p$-rank of an elementary abelian $p$-group of order $p^k$ is $k$.  The following easy result should be known.

\begin{lem}\label{lem: involution centralizers in factor groups} Let $X$ be a finite group and let $Z$ be a subgroup of $\Center(X)$. Let $i\in X$ be a non-central involution. Let $r$ be the $2$-rank of $Z$ and let $\overline{X}=X/Z$. Let $T\leq X$ be the full inverse image of $\Centralizer_{\overline{X}}(\overline{i})$. Then $T/\Centralizer_X(i)$ is an elementary abelian $2$-group of order at most $2^r$ and $|i^X|_2\leq 2^r\cdot |\overline{i}^{\overline{X}}|_2$.
\end{lem}

\begin{proof} Let $g\in T$. Then $i^g=iz$ for some $z\in Z$. As $i^2=1$ and $iz=zi$, we have $z^2=1$. Now $\Centralizer_X(i)^g=\Centralizer_X(i^g)=\Centralizer_X(iz)=\Centralizer_X(i)$, so $\Centralizer_X(i)\unlhd T.$ Moreover, $i^{g^2}=(iz)^g=i^gz=i$, hence $g^2\in\Centralizer_X(i)$. Thus $T/\Centralizer_X(i)$ is an elementary abelian $2$-group.
 Let $\Omega=\{x\in Z: x^2=1\}$. Clearly $\Omega$ is an elementary abelian $2$-subgroup of $X$ of order $2^r$. For each $h\in T$, there exists $z\in \Omega$ such that $i^h=iz$. Moreover, if $i^h=i^k$ for some $h,k\in T$, then $h\Centralizer_X(i)=k\Centralizer_X(i)$. Thus there is an injective map from the set of left cosets of $\Centralizer_X(i)$ in $T$ to $\Omega$ and hence $|T:\Centralizer_X(i)|\leq 2^r$.
Now  $$|i^X|_2=|X:T|_2\cdot |T:\Centralizer_X(i)|_2=|\overline{i}^{\overline{X}}|_2\cdot  |T:\Centralizer_X(i)|\leq 2^r\cdot |\overline{i}^{\overline{X}}|_2.$$ The proof is complete. 
\end{proof}

We next determine all quasi-simple groups which have an involution $i$ whose centralizer is solvable. 
Notice the isomorphisms $\PSU_4(2)\cong\PSp_4(3),\PSL_2(7)\cong\PSL_3(2),\Alt_5\cong\PSL_2(4)\cong\PSL_2(5),\Alt_6\cong\PSL_2(9)$ and $\Alt_8\cong\PSL_4(2)$.

\begin{lem}\label{lem:solvable involution centralizer} Let $L$ be a quasi-simple group and let $S=L/\Center(L)$. Suppose that $L$ has an involution $i$ such that $\Centralizer_L(i)$ is solvable. Then one of the following holds.

\begin{itemize}
\item[$(i)$] $L\cong \rm{M}_{11}$  or $S\cong  \rm{M}_{12}, \rm{M}_{22},\rm{Fi}_{22}$.
\item[$(ii)$]  $L\cong\Alt_n$  $(5\leq n\leq 12)$, $2\cdot \Alt_n$ $(8\leq n\le 12)$ or $3\cdot \Alt_n$ $(6\leq n\leq 7)$.
\item[$(iii)$] $S\cong \PSL_2(q),\PSL_3(q),\PSU_3(q),\Sp_4(q),{}^2{\rm{B}}_2(q)$ with $q=2^f$.

\item[$(iv)$] $S\cong \PSL_n(2)$ $(n=4,5,6),$ $\PSU_n(2)$ $(4\leq n\leq 9)$, $\Sp_{2n}(2)$ $(3\leq n\leq 5)$, $\Omega_{2n}^\pm(2)$ $(4\leq n\leq 5)$, ${}^3\rm{D}_4(2),{}^2\rm{F}_4(2)',\rm{F}_4(2),{}^2\rm{E}_6(2)$.                             
\item[$(v)$] $S\cong \PSL_3(3),\PSL_4(3),\PSU_3(3),\PSU_4(3),\PSp_4(3),\Omega_7(3),\rm{P}\Omega_8^+(3),\rm{G}_2(3)$.
\item[$(vi)$] $L\cong\PSL_2(q)$ with $q$ odd.
\end{itemize}

\end{lem}

\begin{proof} Let $\overline{L}=L/\Center(L)$. Since  $\Centralizer_L(i)$ is solvable,  $i\not\in\Center(L)$
and  $\Centralizer_{\overline{L}}(\overline{i})$ is solvable by Lemma \ref{lem: involution centralizers in factor groups}. Thus $S$ is a non-abelian simple group with a solvable involution centralizer.

(1) Assume that $S$ is a sporadic simple group. The information on the centralizers of involutions of $L$ can be read off from Tables $5.3(a)-(z)$ in \cite{GLS3}. It follows that $L\cong \rm{M}_{11}$ or $S\cong \rm{M}_{12}, \rm{M}_{22},\rm{Fi}_{22}$.

(2) Assume $S\cong\Alt_n$ with $n\ge 5.$  We know that every involution $j$ of $S\cong \Alt_n$ is a product of $r$ disjoint transpositions where $r$ is even. Let $s=n-2r$. By \cite[Proposition 5.2.8]{GLS3} $\Centralizer_S(j)\cong (H_1\times H_2)\langle t\rangle$, where $H_2\cong \Alt_s$ and $H_1\cong R_1L_1$ with $L_1\cong\Sym_r$ and $R_1\cong C_2^{r-1}$. Hence if $r$ or $s$ is at least $5$, then $\Centralizer_S(j)$ is non-solvable. Thus both $r$ and $s$ are at most $4$ and hence $n\leq 12$.  Therefore, if $L=\Alt_n$ with $n\ge 5$, then $5\leq n\leq 12$.  If $L\cong 2\cdot \Alt_n$, then $n\leq 12$ by our observation above. However, $2\cdot\Alt_n$ has a non-central involution only when $n\ge 8.$ Thus if $L\cong 2\cdot \Alt_n$, then $8\leq n\leq 12.$  For $n=6,7$, only $3\cdot \Alt_n$ has a solvable involution centralizer.

(3) Assume $S$ is a finite simple group of Lie type. The centralizers of involutions of $L$ are determined in \cite{AS} and Tables $4.5.1$ and $4.5.2$ in \cite{GLS3}. From these results, it is easy to get all the possibilities for $L$. (See also \cite[Lemma 3.4]{LO}). 
\end{proof}

We use the convention that  $\PSL_n^\epsilon(q)$ is $\PSL_n(q)$ if $\epsilon=+$  and $\PSU_n(q)$ if $\epsilon=-$. A similar convention applies to $\SL_n^\epsilon(q)$. We now prove the main result of this section.

\begin{table}[ht]
\caption{Some small simple groups}
\begin{center}
\begin{tabular}{l|ccccc}
\hline
$S$&$e_2(S)$&$z$&$\nu_2(|z^S|)$&$|\Out(S)|$&$\Mult(S)$  \\\hline

$\textrm{M}_{12}$&$2$&$3a$&$5$&$2$&$\bbZ_2$\\

$\textrm{M}_{22}$&$0$&$5a$&$7$&$2$&$\bbZ_{12}$\\

$\textrm{Fi}_{22}$&$1$&$3d$&$17$&$2$&$\bbZ_6$\\

$\textrm{A}_{8}$&$1$&$5a$&$6$&$2$&$\bbZ_2$\\

$\textrm{A}_{9}$&$1$&$3b$&$4$&$2$&$\bbZ_2$\\

$\textrm{A}_{10}$&$1$&$3c$&$7$&$2$&$\bbZ_2$\\

$\textrm{A}_{11}$&$1$&$3c$&$6$&$2$&$\bbZ_2$\\

$\textrm{A}_{12}$&$0$&$3d$&$7$&$2$&$\bbZ_2$\\

$\PSL_3(4)$&$0$&$3a$&$6$&$12$&$\bbZ_4\times \bbZ_{12}$\\

${}^2\textrm{B}_{2}(8)$&$0$&$5a$&$6$&$3$&$\bbZ_2\times\bbZ_2$\\

$\PSU_{4}(2)$&$1$&$5a$&$6$&$2$&$\bbZ_2$\\

$\PSU_{6}(2)$&$3$&$3c$&$12$&$6$&$\bbZ_2\times\bbZ_6$\\

$\textrm{Sp}_{6}(2)$&$2$&$3c$&$7$&$1$&$\bbZ_2$\\

${\Omega}_{8}^+(2)$&$2$&$3d$&$9$&$6$&$\bbZ_2\times\bbZ_2$\\

$\textrm{F}_{4}(2)$&$4$&$3c$&$18$&$2$&$\bbZ_2$\\

${}^2\textrm{E}_{6}(2)$&$5$&$3c$&$27$&$6$&$\bbZ_2\times\bbZ_6$\\

$\textrm{PSU}_4(3)$&$0$&$3d$&$7$&$8$&$\bbZ_3\times\bbZ_{12}$\\

$\Omega_7(3)$&$0$&$3g$&$8$&$2$&$\bbZ_6$\\

$\textrm{P}\Omega_{8}^+(3)$&$2$&$3m$&$11$&$24$&$\bbZ_2\times\bbZ_2$\\

\hline
\end{tabular}
\end{center}
\label{tab}
\end{table}%

\begin{thm}\label{th:quasisimple groups}
Let $L$ be a quasi-simple group with center $Z$ and let $S=L/Z$. Suppose that $Z$ is a $2$-group and the following conditions hold.

\begin{itemize}
\item[$(1)$] If $i$ is a non-central involution of $L$, then $\Centralizer_L(i)$ is $2$-closed.
\item[$(2)$] If $x,z\in \Real(L)\setminus \Center(L)$ are  non-central real elements, where $x$ is a $2$-element, then $|z^L|_2\leq |\Out(S)|_2\cdot |x^L|_2$.
\end{itemize}
Then $L\cong\SL_2(q)$ with $q\ge 5$ odd.
\end{thm}

\begin{proof} We consider the following cases.

\textbf{Case 1}. All involutions of $L$ are central. By the main theorem in \cite{Griess}, $L\cong\SL_2(q)$ with $q\ge 5$ odd or $2\cdot \Alt_7$. If the first case holds, then we are done. So, assume that $L\cong 2\cdot\Alt_7$. We have $|\Out(S)|_2=2$. Using \cite{Atlas}, we can check that $L$ has two real elements $z$ and $x$ of order $3$ and $4$ respectively, with $|z^L|=280$ and $|x^L|=210$. Clearly $|z^L|_2=2^3>|\Out(S)|_2\cdot|x^L|_2=2^2$, violating condition $(2).$ So this case cannot occur.

\medskip
\textbf{Case 2}. $\Center(L)$ is trivial. Then $L$ is a non-abelian simple group.  Let $x\in L$ be a $2$-central involution of $L$, i.e., $x$ is an involution that lies in the center of some Sylow $2$-subgroup of $L$. We have $|x^L|_2=1$ so $(2)$ implies $|z^L|_2\leq |\Out(L)|_2$ for all non-central real elements $z\in L$. 

The centralizer of every non-central involution of $L$ is $2$-closed by the hypothesis. Since $L$ is simple, it follows that the centralizer of every involution of $L$ is $2$-closed. Now  \cite[Theorem 1]{Suzuki} yields that $L$ is isomorphic to one of the following groups: $\Alt_6$, $\PSL_2(p)$, $p$ a Fermat or a Mersenne prime, $\PSL_2(2^f)$ with $f\ge 2$, $\PSL_3(q),\PSU_3(q)$ with $q=2^f$ or  ${}^2{\rm{B}}_2(2^{2f+1})$ with $f\ge 1.$

Assume that $L\cong {}^2{\rm{B}}_2(2^{2f+1})$ with $f\ge 1.$ It follows from Propositions 3 and 16 in \cite{Suzuki-1} that $L$ has a real element $z$ of order $2^f-1$ with $|\Centralizer_L(z)|=2^f-1$ and so $|z^L|_2=2^{2(2f+1)}\ge 64$. Since $|\Out(L)|=2f+1$ is odd, $|z^L|_2>|\Out(L)|_2$, hence this case cannot occur. 

Assume $L\cong\PSL_2(p)$ with $p$ a Fermat or a Mersenne prime. We have $|\Out(L)|=2$ and $L$ possesses a real element $z$ of odd order $(p+\delta)/2$, where $p\equiv \delta$ (mod $4$) and $|z^L|_2=|L|_2\ge 4$. Since $|z^L|_2>2=|\Out(L)|_2$, this case cannot happen.

 If $L\cong\Alt_6$, then $|\Out(L)|=2^2$. However, $\Alt_6$ has a real element $z$ of order $3$ with $|z^L|=40$ and thus $|z^L|_2=8>|\Out(S)|_2=4.$ Thus this case cannot happen. 
 
 Next, if $L\cong\PSL_2(2^f)$ with $f\ge 2,$ then $L$ has a real element $z$ of order $2^f-1$ with $|z^L|=2^f(2^f+1).$ Clearly $|z^L|_2=2^f>f\ge|\Out(L)|_2$ as $|\Out(L)|=f$.

Finally, assume $L\cong\PSL_3^\epsilon(2^f)$. We have $f\ge 2$ as $\PSL_3(2)\cong\PSL_2(7)$, where $7=2^3-1$ is a Mersenne prime and $\PSU_3(2)$ is not simple. In both cases, $|\Out(L)|=2df$ with $d=(3,2^f-\epsilon1)$ so $|\Out(L)|_2=2^{1+\nu_2(f)}$. The quasi-simple group $X=\SL_3^\epsilon(2^f)$ possesses real elements $h$ of order $2^f+\epsilon1$ with $|\Centralizer_X(h)|=4^f-1$ and $g$ of order $2^f-\epsilon1$  with $|\Centralizer_X(g)|=(2^f-\epsilon1)^2$ (see \cite[Lemma 4.4 (3)]{GNT}). Now let $y\in\{g,h\}$ be an element with $o(y)$ relatively prime to  $d=(3,2^f-\epsilon1)$ and let $z$ be the image of $y$ in $L=\PSL_3^\epsilon(2^f)\cong X/\Center(X)$. Then $|\Centralizer_L(z)|=|\Centralizer_X(y)/\Center(X)|$ (by \cite[Lemma 7.7]{Isaacs-1}) is odd, so $|z^L|_2=|L|_2=2^{3f}$.   Since $f\ge 2,$  $2^{3f}>2^{2f}\ge 2^{1+\nu_2(f)}$. Hence these cases cannot occur.

\medskip
\textbf{Case 3}. $\Center(L)$ is nontrivial and $L$ has a non-central involution. Let $j$ be a non-central involution of $L$. By condition (1), $\Centralizer_L(j)$ is $2$-closed and thus it is solvable. Hence $L$ is one of the quasi-simple groups in Lemma \ref{lem:solvable involution centralizer}. Moreover, as $\Center(L)$ is a nontrivial $2$-group, the Schur multiplier $\Mult(S)$ of $S\cong L/\Center(L)$ is of even order. It follows that we only need to consider the following cases. 
\begin{itemize}
\item[$(i)$] $S\cong  \rm{M}_{12}, \rm{M}_{22},\rm{Fi}_{22}$, $\Alt_n$ $(8\leq n\le 12)$, $\PSL_3(4),{}^2{\rm{B}}_2(8)$.

\item[$(ii)$] $S\cong \PSU_n(2)$ $(n=4,6)$, $\Sp_{6}(2)$, $\Omega_{8}^+(2),\rm{F}_4(2),{}^2\rm{E}_6(2)$.                             
\item[$(iii)$] $S\cong\PSU_4(3),\PSp_4(3),\Omega_7(3),\rm{P}\Omega_8^+(3)$.
\end{itemize}

We will show that these cases cannot occur by showing that condition (2) does not hold. Clearly $\overline{j}\in \overline{L}$ is a non-central involution and  by Lemma \ref{lem: involution centralizers in factor groups}, $|j^L|_2\leq 2^{r_2(\Center(L))}\cdot |\overline{j}^{\overline{L}}|_2$, where $r_2(\Center(L))$ is the $2$-rank of $\Center(L)$. Let $\tilde{S}$ be the perfect central extension of $S$ such that $\tilde{S}/\Center(\tilde{S})\cong S$ and $|\Center(\tilde{S})|=|\Mult(S)|$. From \cite{Atlas}, we see that $\Mult(S)$ can be written as a direct product of at most two (possibly trivial) cyclic groups, so $|r_2(\Center(L))|\leq r_2(\Mult(S))\leq 2$. Let ${e_2(S)}=\textrm{max}\{\nu_2(|x^S|):\text{ $x$ is an involution in $S$}\}$. Then  for each non-central involution $j\in L$, we have $\nu_2(|j^L|)\leq {r_2(\Mult(S))}+e_2(S).$

Let $z\in S$ be a nontrivial real element of odd order. There exists $y\in \Real(L)\setminus\Center(L)$ of odd order such that $z$ is the image of $y$ in $L/\Center(L)\cong S$ (see \cite[Lemma 3.2]{NST} or \cite[Lemma 2.2]{GNT}) and by applying \cite[Lemma 7.7]{Isaacs-1}, $|z^S|=|y^L|$ (noting $(o(y),|\Center(L)|)=1$). Hence to show that condition (2) does not hold, it suffices to find a nontrivial real element $z\in S$ of odd order such that the following inequality holds.

\begin{equation}\label{eqn1}\nu_2(|z^S|)>\nu_2(|\Out(S)|)+r_2(\Mult(S))+e_2(S).\end{equation}
 For each simple group $S$ in $(i)-(iii)$ above,  we list in Table \ref{tab} the invariant $e_2(S)$, the largest $2$-part of the sizes of conjugacy classes of involutions of $S$ in the second column, the Atlas class name and the $2$-part of the conjugacy class of odd order real element $z\in S$, the order of the outer automorphism group $\Out(S)$ and in the last column the Schur multiplier $\Mult(S)$. The information in this table can be read off from the character table of $S$ using \cite{Atlas}.

From Table \ref{tab}, we can check that Equation \eqref{eqn1} holds  which completes our proof.
\end{proof}

\section*{Acknowledgment}  The author is grateful to Kay Magaard, Ben Brewster and Dan Rossi for their help during the preparation of this paper. He also thanks the anonymous referee for his or her comments, which helped to improve the clarity of this paper.

\end{document}